\documentclass[11pt]{amsart}
\usepackage{amsmath,amssymb,latexsym,cancel,rotating}

\numberwithin{equation}{section}

\newtheorem{theorem}{Theorem}[section]
\newtheorem{proposition}[theorem]{Proposition}

\newtheorem{corollary}[theorem]{Corollary}
\newtheorem{lemma}[theorem]{Lemma}

\theoremstyle{definition}

\newtheorem{remark}[theorem]{Remark}

\newtheorem{definition}[theorem]{Definition}

\input xy
\xyoption{poly}
\xyoption{2cell}
\xyoption{all}

\def\ZZ{\mathbb{Z}}

\def\Acal{\mathcal{A}}

\def\Fcal{\mathcal{F}}

\def\Tcal{\mathcal{T}}

\def\Hom{\text{Hom}}
\def\Ext{\text{Ext}}

\renewcommand{\eqref}[1]{{\rm (\ref{#1})}}

\begin{document}

\title[The multiplication theorem and bases in quantum
cluster algebras] {The multiplication theorem and bases in finite
and affine quantum cluster algebras}

\author{Ming Ding and Fan Xu}
\address{Institute for advanced study\\
Tsinghua University\\
Beijing 100084, P.~R.~China} \email{m-ding04@mails.tsinghua.edu.cn
(M.Ding)}
\address{Department of Mathematical Sciences\\
Tsinghua University\\
Beijing 100084, P.~R.~China} \email{fanxu@mail.tsinghua.edu.cn
(F.Xu)}

%\keywords{$\mathbb{Z}[q^{\pm \frac{1}{2}}]-$basis, quantum cluster
%algebra.}

%\date{April 10, 2010}
\thanks{The research was
supported  by NSF of China (No. 11071133)}


\maketitle

\begin{abstract} We prove a multiplication
theorem for quantum cluster algebras of acyclic quivers. The theorem
generalizes the multiplication formula for quantum cluster variables
in \cite{fanqin}. Moreover some $\mathbb{ZP}$-bases in quantum
cluster algebras of finite and affine types are constructed.  Under
the specialization $q$ and coefficients to $1$, these bases are the
integral bases of cluster algebra of finite and affine types (see
\cite{CK1} and \cite{DXX}).
\end{abstract}

\section{Introduction}
Quantum cluster algebras  were introduced by A.~Berenstein and
A.~Zelevinsky \cite{berzel} as a noncommutative analogue of cluster
algebras  \cite{ca1}\cite{ca2} to  study  canonical bases.  A
quantum cluster algebra   is generated by a  set of generators
called the \textit{quantum cluster variables} inside an ambient
skew-field $\mathcal{F}$. Under the specialization $q=1,$ the
quantum cluster algebras are exactly cluster algebras which were
introduced by S.~Fomin and A.~Zelevinsky \cite{ca1}\cite{ca2}.

Cluster algebras have a close link to quiver representations via
cluster categories invented in \cite{BMRRT}.  The link is explicitly
characterized by the Caldero-Chapoton map (\cite{caldchap}) and the
Caldero-Keller multiplication theorems (\cite{CK1},\cite{CK2}). The
Caldero-Chapoton map associates the objects in the cluster
categories to some Laurent polynomials, in particular, sends rigid
objects to cluster variables. The Caldero-Keller multiplication
theorems show the multiplication rules between images of objects
under the Caldero-Chapoton map. The theorem is remarkable. On the
one hand, it is similar to the multiplication in a dual Hall algebra
and unifies homological and geometric properties of cluster
categories and combinatorial properties of cluster algebras. On the
other hand, since cluster algebras were introduced to study
canonical bases, it is important to construct integral bases of
cluster algebras. The Caldero-Keller  multiplication theorems are
essentially important to construct integral bases of cluster
algebras. Following this link, some good bases have been constructed
for finite and affine cluster algebras (\cite{CK1}, \cite{calzel},
\cite{Dup} and \cite{DXX}).

Naturally, one can study the quantum analogue of the link. Recently,
Rupel (\cite{rupel}) defined a quantum analog of the
Caldero-Chapoton map (called the quantum Caldero-Chapoton map) and
conjectured that cluster variables could be expressed as images of
indecomposable rigid objects under the quantum Caldero-Chapoton
formula. A key ingredient of the conjecture is to confirm the
multiplication rules between quantum cluster variables given by
\cite{berzel}. Most recently, the conjecture has been proved by Qin
(\cite{fanqin}) for acyclic equally valued quivers. There Qin
constructed a quantum cluster multiplication formula and then
confirmed the multiplication rules between quantum cluster
variables.

The present paper is contributed to  prove a multiplication theorem
(a combination of Theorem \ref{multi-formula} and \ref{exchange2})
for acyclic quantum cluster algebras in Section 3. The theorem
generalizes the quantum cluster multiplication formula in
\cite{fanqin} and can be viewed as a quantum analogue of the
$1$-dimensional Caldero-Keller multiplication theorem in \cite{CK2}.
Compared to the role which the Caldero-Keller multiplication
theorems play for cluster algebras, our multiplication theorem is
worthy of highlighting and also reflects the information and the
difficulty to prove the more general quantum analog of the
Caldero-Keller multiplication theorems. The main idea in the proof
of the multiplication theorem is taken from \cite{Hubery}. Moreover,
we construct some good $\mathbb{ZP}$-bases in quantum cluster
algebras of finite and affine types.  By specializing $q$ and
coefficients to $1$, these bases induce the  good bases for cluster
algebras of finite\cite{CK1} and affine types\cite{DXX},
respectively.

\section{The quantum Caldero-Chapoton map}
\subsection{Quantum cluster algebras} The main reference for quantum cluster algebras is \cite{berzel}.
Here, we also recommend \cite[Section2]{fanqin} as a nice reference. Let $L$ be a lattice of rank $m$ and
$\Lambda:L\times L\to \ZZ$ a skew-symmetric bilinear form. Let $q$
be a formal variable and consider the ring of integer Laurent
polynomials $\ZZ[q^{\pm1/2}]$.  Define the \textit{based quantum
torus} associated to the pair $(L,\Lambda)$ to be the
$\ZZ[q^{\pm1/2}]$-algebra $\mathcal{T}$ with a distinguished
$\ZZ[q^{\pm1/2}]$-basis $\{X^e: e\in L\}$ and the  multiplication
given by
\[X^eX^f=q^{\Lambda(e,f)/2}X^{e+f}.\]
It is easy to see  that $\Tcal$ is associative and the basis
elements satisfy the following relations:
\[X^eX^f=q^{\Lambda(e,f)}X^fX^e,\ X^0=1,\ (X^e)^{-1}=X^{-e}.\]  It is known that $\Tcal$ is an Ore domain, i.e.,   is contained in its
skew-field of fractions $\Fcal$.  The quantum cluster algebra
 will be defined as a
$\ZZ[q^{\pm1/2}]$-subalgebra of $\Fcal$.

A \textit{toric frame} in $\Fcal$ is a map $M: \ZZ^m\to \Fcal
\setminus \{0\}$ of the form \[M({\bf c})=\varphi(X^{\eta({\bf
c})})\] where $\varphi$ is an automorphism of $\Fcal$ and $\eta:
\ZZ^m\to L$ is an  isomorphism of lattices.  By the definition, the
elements $M({\bf c})$ form a $\ZZ[q^{\pm1/2}]$-basis of the based
quantum torus $\Tcal_M:=\varphi(\Tcal)$ and satisfy the following
relations:
\[M({\bf c})M({\bf d})=q^{\Lambda_M({\bf c},{\bf d})/2}M({\bf c}+{\bf d}),\
M({\bf c})M({\bf d})=q^{\Lambda_M({\bf c},{\bf d})}M({\bf d})M({\bf
c}),\]
\[ M({\bf 0})=1,\ M({\bf c})^{-1}=M(-{\bf c}),\]
where $\Lambda_M$ is the skew-symmetric bilinear form on $\ZZ^m$
obtained from the lattice isomorphism $\eta$.  Let $\Lambda_M$ also
denote the skew-symmetric $m\times m$ matrix defined by
$\lambda_{ij}=\Lambda_M(e_i,e_j)$ where $\{e_1, \ldots, e_m\}$ is
the standard basis of $\ZZ^m$.  Given a toric frame $M$, let
$X_i=M(e_i)$.  Then we have
$$\Tcal_M=\ZZ[q^{\pm1/2}]\langle X_1^{\pm 1}, \ldots,
X_m^{\pm1}:X_iX_j=q^{\lambda_{ij}}X_jX_i\rangle.$$  An easy
computation shows that
\[M({\bf c})=q^{\frac{1}{2}\sum_{i<j}
c_ic_j\lambda_{ji}}X_1^{c_1}X_2^{c_2}\cdots X_m^{c_m}=:X^{{\bf c}} \
\ \ ({\bf c}\in\ZZ^m).\]

Let $\Lambda$ be an $m\times m$ skew-symmetric matrix and let
$\widetilde{B}$ be an $m\times n$ matrix for some positive integer
$n\leq m$. We call the pair $(\Lambda, \widetilde{B})$
\textit{compatible} if $\widetilde{B}^T\Lambda=(D|0)$ is an $n\times
m$ matrix with $D=diag(d_1,\cdots,d_n)$ where $d_i\in \mathbb{N}$
for $1\leq i\leq n$. The pair $(M,\widetilde{B})$ is called a
\textit{quantum seed} if the pair $(\Lambda_M, \widetilde{B})$ is
compatible.  Define the $m\times m$ matrix $E=(e_{ij})$ by
\[e_{ij}=\begin{cases}
\delta_{ij} & \text{if $j\ne k$;}\\
-1 & \text{if $i=j=k$;}\\
max(0,-b_{ik}) & \text{if $i\ne j = k$.}
\end{cases}
\]
For $n,k\in\ZZ$, $k\ge0$, denote ${n\brack
k}_q=\frac{(q^n-q^{-n})\cdots(q^{n-k+1}-q^{-n+k-1})}{(q^k-q^{-k})\cdots(q-q^{-1})}$.
Let ${\bf c}=(c_1,\ldots,c_m)\in\ZZ^m$ with $c_{k}\geq 0$.  Define
the toric frame $M': \ZZ^m\to \Fcal \setminus \{0\}$ as follows:
\begin{equation}\label{eq:cl_exp}M'({\bf c})=\sum^{c_k}_{p=0} {c_k \brack p}_{q^{d_k/2}} M(E{\bf c}+p{\bf b}^k),\ \ M'({\bf -c})=M'({\bf c})^{-1}.\end{equation}
where the vector ${\bf b}^k\in\ZZ^m$ is the $k-$th column of
$\widetilde{B}$.    Then the quantum seed $(M',\widetilde{B}')$ is
defined to be the mutation of $(M,\widetilde{B})$ in direction $k$.
In general, two quantum seeds $(M, \widetilde{B})$ and $(M',
\widetilde{B}')$ are mutation-equivalent if they can be obtained
from each other by a sequence of mutations, denoted by $(M,
\widetilde{B})\sim (M', \widetilde{B}')$. Let $\mathcal{C}=\{M'(e_i)
\mid (M, \widetilde{B})\sim (M', \widetilde{B}'), i=1, \cdots n\}$.
The elements of $\mathcal{C}$ are called \textit{quantum cluster
variables}. Let $\mathcal{P}=\{M(e_i): i=n+1, \cdots, m]\}$ and  the
elements of $\mathcal{P}$ are called \emph{coefficients}.  Given
$(M', \widetilde{B}')\sim (M, \widetilde{B})$ and ${\bf c}=(c_i)\in
\mathbb{Z}^m$, a element $M'(\bf c)$ is called  a \emph{quantum
cluster monomial} if $c_i\geq 0$ for $i=1, \cdots, n$ and $0$ for
$i=n+1, \cdots, m.$  We denote by $\mathbb{P}$ the multiplicative
group by $q^{\frac{1}{2}}$ and $\mathcal{P}$. Write $\mathbb{ZP}$ as
the ring of Laurent polynomials in the elements of $\mathcal{P}$
with coefficients in $\mathbb{Z}[q^{\pm 1/2}]$.  The \textit{quantum
cluster algebra} $\Acal_q(\Lambda_M,\widetilde{B})$ is the
$\mathbb{ZP}$-subalgebra of $\Fcal$ generated by $\mathcal{C}$. We
associate $(M,\tilde{B})$ a $\ZZ$-linear \emph{bar-involution} on
$\Tcal_M$ defined by
\[\overline{q^{r/2}M({\bf c})}=q^{-r/2}M({\bf c}), \ \  (r\in\ZZ,\
{\bf c}\in\ZZ^n).\]
 It is easy to show that
$\overline{XY}=\overline{Y}~\overline{X}$ for all $X,Y\in
\Acal_q(\Lambda_M, \widetilde{B})$ and that each element of
$\mathcal{C}\cup \mathcal{P}$ is \emph{bar-invariant}.

Now assume that there exists a finite field $k$ satisfying $|k|=q$.
In the same way, we can define based quantum torus
$\mathcal{T}_{|k|}$ and \emph{specialized quantum cluster algebras}
$\Acal_{|k|}(\Lambda_M,\widetilde{B})$ by substituting
$\mathbb{Z}[|k|^{\pm\frac{1}{2}}]$ for
$\mathbb{Z}[q^{\pm\frac{1}{2}}]$ in the above definition. By
\cite[Corollary 5.2]{berzel}, $\Acal_q(\Lambda_M, \widetilde{B})$
and $\Acal_{|k|}(\Lambda_M,\widetilde{B})$ are subalgebras of
$\mathcal{T}$ and $\mathcal{T}_{|k|}$, respectively. There is a
specialization map $ev: \mathcal{T}\rightarrow \mathcal{T}_{|k|}$ by
mapping $q^{\frac{1}{2}}$ to $|k|^{\frac{1}{2}}$, which induces a
bijection between quantum monomials of
$\Acal_{q}(\Lambda_M,\widetilde{B})$ and
$\Acal_{|k|}(\Lambda_M,\widetilde{B})$ (\cite[Section 2.2]{fanqin}).

\subsection{The quantum Caldero-Chapoton map}Let $k$ be a finite field with cardinality $|k|=q$ and
$m\geq n$ be two positive integers and $\widetilde{Q}$ an acyclic
quiver with vertex set $\{1,\ldots,m\}$ \cite{fanqin}. Denote the
subset $\{n+1,\dots,m\}$ by $C$. The elements in $C$ are called the
\emph{frozen vertices }, and $\widetilde{Q}$ is called an \emph{ice
quiver}. The full subquiver $Q$ on the vertices $1,\ldots,n$ is
called the \emph{principal part} of $\widetilde{Q}$.

Let $\widetilde{B}$ be the $m\times n$ matrix associated to the ice
quiver $\widetilde{Q}$, i.e., its entry in position $(i,j)$ is
\[
b_{ij}=|\{\mathrm{arrows}\, i\longrightarrow
j\}|-|\{\mathrm{arrows}\, j\longrightarrow i\}|
\]
for $1\leq i\leq m$, $1\leq j\leq n$. And let $\widetilde{I}$ be the
left $m\times n$ matrix of the identity matrix of size $m\times m$.
Further assume that there exists some antisymmetric $m\times m$
integer matrix $\Lambda$ such that
\begin{align}\label{eq:simply_laced_compatible}
\Lambda(-\widetilde{B})=\widetilde{I}:=\begin{bmatrix}I_n\\0
\end{bmatrix},
\end{align}
where $I_n$ is the identity matrix of size $n\times n$. Thus, the
matrix $\widetilde{B}$ is of full rank.

Let $\widetilde{R}$ and $\widetilde{R}^{tr}$ be the $m\times n$
matrix with its entry in position $(i,j)$ is
\[
\widetilde{r}_{ij}=\mathrm{dim}_{k}\mathrm{Ext}^{1}_{k\widetilde{Q}}(S_i,S_j)
\]
and
\[
\widetilde{r}^{*}_{ij}=\mathrm{dim}_{k}\mathrm{Ext}^{1}_{k\widetilde{Q}}(S_j,S_i)
\]
for $1\leq i\leq m$, $1\leq j\leq n$, respectively. Note that
\[
\mathrm{dim}_{k}\mathrm{Ext}^{1}_{k\widetilde{Q}}(S_i,S_j)=|\{\mathrm{arrows}\,
j\longrightarrow i\}|.
\]
 Denote the principal parts of
the matrices $\widetilde{B}$ and $\widetilde{R}$ by $B$ and $R$
respectively. Note that
$\widetilde{B}=\widetilde{R}^{tr}-\widetilde{R}$ and $B=R^{tr}-R$
where $R^{tr}$ represents the transposition of the matrix $R.$ In
general,  the matrix $B$ is not of full rank so that there exists no
matrix $\Lambda$ compatible with $B$. Hence, one need add  some
frozen vertices to $Q$ and then obtain an acyclic quiver
$\widetilde{Q}$ with a compatible pair $(\widetilde{B}, \Lambda).$

Let $\mathcal C_{\widetilde{Q}}$ be the cluster category of $k
\widetilde{Q}$, i.e., the orbit category of the derived category
$\mathcal{D}^b(\widetilde{Q})$ by the functor $F=\tau\circ[-1]$
where $\tau$ is the Auslander-Reiten translation and $[1]$ is the
translation functor. We note that the indecomposable objects of the
cluster category $\mathcal C_{\widetilde{Q}}$ are either the
indecomposable $k \widetilde{Q}$-modules or $P_i[1]$ for
indecomposable projective modules $P_i$($1\leq i \leq m$).  Each
object $M$ in $\mathcal C_{\widetilde{Q}}$ can be uniquely
decomposed in the following way:
$$M\cong M_0\oplus P_M[1]$$
where $M_0$ is a $k\widetilde{Q}$-module and $P_M$ is a projective
$k\widetilde{Q}$-module. Let $P_M=\bigoplus_{1\leq i \leq m}m_iP_i.$
We extend the definition of the dimension vector
$\mathrm{\underline{dim}}$ on modules in $\mathrm{mod}k
\widetilde{Q}$ to objects in $\mathcal C_{\widetilde{Q}}$ by setting
$$\mathrm{\underline{dim}}M=\mathrm{\underline{dim}}M_0-(m_i)_{1\leq i \leq m}.$$
The Euler form on $k\widetilde{Q}$-modules $M$ and $N$ is given by
$$\langle M,N\rangle=\mathrm{dim}_{k}\mathrm{Hom}_{k\widetilde{Q}}(M,N)-\mathrm{dim}_{k}\mathrm{Ext}^{1}_{k\widetilde{Q}}(M,N).$$
Note that the Euler form only depends on the dimension vectors of
$M$ and $N$. As in \cite{Hubery}, we define
$$
[M, N]=\mathrm{dim}_{k}\mathrm{Hom}_{k\widetilde{Q}}(M,N)\mbox{ and
}[M, N]^1=\mathrm{dim}_{k}\mathrm{Ext}^{1}_{k\widetilde{Q}}(M,N).
$$

The quantum Caldero-Chapoton map of an acyclic quiver
$\widetilde{Q}$ has been studied in \cite{rupel} and \cite{fanqin}.
Here, we reformulate their definitions to the following map
$$X^{\widetilde{Q}}_?: \mathrm{obj}\mathcal C_{\widetilde{Q}}\longrightarrow \Tcal$$
defined by the following rule: If $M$ is a $k Q$-module and $P$ is a
projective $k \widetilde{Q}$-module, then
                    $$
                       X^{\widetilde{Q}}_{M\oplus P[1]}=\sum_{\underline{e}} |\mathrm{Gr}_{\underline{e}} M|q^{-\frac{1}{2}
\langle
\underline{e},\underline{m}-\underline{e}\rangle}X^{\widetilde{B}\underline{e}-(\widetilde{I}-\widetilde{R})\underline{m}+\underline{\mathrm{dim}}
P/\mathrm{rad}P},
                    $$
where $\underline{\mathrm{dim}} M= \underline{m}$ and
$\mathrm{Gr}_{\underline{e}}M$ denotes the set of all submodules $V$
of $M$ with $\underline{\mathrm{dim}} V= \underline{e}$. Usually, we
omit the upper index $\widetilde{Q}$ in the notation
$X^{\widetilde{Q}}_?$ (except Section 4 and Section 5) if there is
no confusion. We note that
$$
X_{P[1]}=X_{\tau P}=X^{\underline{\mathrm{dim}} P/rad
P}=X^{\underline{\mathrm{dim}}\mathrm{soc}I}=X_{I[-1]}=X_{\tau^{-1}I}.
$$
for any projective $k\widetilde{Q}$-module $P$ and injective
$k\widetilde{Q}$-module $I$ with $\mathrm{soc}I=P/\mathrm{rad}P.$
Hereinafter, we denote by the corresponding underlined small letter
$\underline{x}$ the dimension vector of a $kQ$-module $X$ and view
$\underline{x}$ as a column vector in $\mathbb{Z}^n.$

\section{Multiplication theorems for acyclic quantum cluster algebras}
Throughout this section, assume that $\widetilde{Q}$ is an acyclic
quiver and $Q$ is its full subquiver. In this section, we will prove
a multiplication theorem for any acyclic quantum cluster algebra.
First, we improve Lemma 5.2.1 and Corollary 5.2.2 in \cite{fanqin},
i.e., here we handle the dimension vector of any $kQ$-module while
in \cite{fanqin} the author only deals with dimension vectors of
rigid modules.
\begin{lemma}\label{1}
For any dimension vector $\underline{m}, \underline{e},
\underline{f}\in \mathbb{Z}^{n}_{\geq 0},$ we have
$$(1)\ \Lambda((\widetilde{I}-\widetilde{R})\underline{m}, \widetilde{B}\underline{e})=-\langle \underline{e}, \underline{m}\rangle;$$
$$(2)\ \Lambda(\widetilde{B}\underline{e}, \widetilde{B}\underline{f})=\langle \underline{e}, \underline{f}\rangle-\langle \underline{f}, \underline{e}\rangle.$$
\end{lemma}
\begin{proof}
By definition, we have \begin{eqnarray}
   && \Lambda((\widetilde{I}-\widetilde{R})\underline{m}, \widetilde{B}\underline{e})  \nonumber\\
   &=& \underline{m}^{tr}(\widetilde{I}-\widetilde{R})^{tr}\Lambda \widetilde{B}\underline{e}=-\underline{m}^{tr}(\widetilde{I}-\widetilde{R})^{tr}\begin{bmatrix}I_n\\0 \end{bmatrix}\underline{e}\nonumber\\
   &=& -\underline{m}^{tr}(I_{n}-R)^{tr}\underline{e}=-\underline{e}^{tr}(I_{n}-R)\underline{m}\nonumber\\
  &=& -\langle \underline{e}, \underline{m}\rangle.\nonumber
\end{eqnarray}
As for (2), the left side of the desired equation is equal to
$$\underline{e}^{tr}\widetilde{B}^{tr}\Lambda
\widetilde{B}\underline{f}=-\underline{e}^{tr}\widetilde{B}^{tr}\begin{bmatrix}I_n\\0
\end{bmatrix}\underline{f}=-\underline{e}^{tr}B^{tr}\underline{f}.$$
The right side is
\begin{eqnarray}
   && \langle \underline{e}, \underline{f}\rangle-\langle \underline{f}, \underline{e}\rangle  \nonumber\\
   &=& \underline{e}^{tr}(I_{n}-R)\underline{f}-\underline{f}^{tr}(I_{n}-R)\underline{e}\nonumber\\
   &=& \underline{e}^{tr}(I_{n}-R)\underline{f}-\underline{e}^{tr}(I_{n}-R)^{tr}\underline{f}\nonumber\\
  &=& \underline{e}^{tr}(R^{tr}-R)\underline{f}=-\underline{e}^{tr}(R-R^{tr})\underline{f}=-\underline{e}^{tr}B^{tr}\underline{f}.\nonumber
\end{eqnarray}
Thus we prove the lemma.
\end{proof}
\begin{corollary}\label{2}
For any dimension vector $\underline{m}, \underline{l},
\underline{e}, \underline{f}\in \mathbb{Z}^{n}_{\geq 0},$ we have
\begin{eqnarray}
   && \Lambda(\widetilde{B}\underline{e}-(\widetilde{I}-\widetilde{R})\underline{m},\widetilde{B}\underline{f}-(\widetilde{I}-\widetilde{R})\underline{l})  \nonumber\\
  &=&\Lambda((\widetilde{I}-\widetilde{R})\underline{m},(\widetilde{I}-\widetilde{R})\underline{l})+\langle \underline{e}, \underline{f}\rangle-\langle \underline{f}, \underline{e}\rangle-\langle \underline{e}, \underline{l}\rangle+\langle \underline{f}, \underline{m}\rangle.\nonumber
\end{eqnarray}
\end{corollary}

For any $kQ-$modules $M,N,E$, denote by $\varepsilon_{MN}^{E}$  the
cardinality of the set $\mathrm{Ext}^{1}_{kQ}(M,N)_{E}$ which is the
subset of $ \mathrm{Ext}^{1}_{kQ}(M,N)$ consisting of those
equivalence classes of short exact sequences with middle term
isomorphic to $M$ (\cite[Section 4]{Hubery}). For $kQ$-modules $M$,
$A$ and $B$, we denote by $F^M_{AB}$ the number of submodules $U$ of
$M$ such that $U$ is isomorphic to $B$ and $M/U$ is isomorphic to
$A$. Then by definition, we have
$$|\mathrm{Gr}_{\underline{e}}(M)|=\sum_{A, B;
\underline{\mathrm{dim}}B=\underline{e}}F_{AB}^M.
$$ Different from the case in cluster categories, for $kQ$-modules,
it does not generally hold  that $X_{N}X_{M}=X_{N\oplus M}.$ We have
the following explicit characterization, which is a generalization
of \cite[Proposition 5.3.2]{fanqin}.
\begin{theorem}\label{hall multi}
Let $M$ and $N$ be $kQ$-modules. Then
$$q^{[M,N]^{1}}X_{N}X_{M}=q^{-\frac{1}{2}\Lambda((\widetilde{I}-\widetilde{R})\underline{m},
(\widetilde{I}-\widetilde{R})\underline{n})}
\sum_{E}\varepsilon_{MN}^{E}X_E.$$
\end{theorem}
\begin{proof}
We apply Green's formula in \cite{Green}
$$\sum_{E}\varepsilon_{MN}^{E}F^{E}_{XY}=\sum_{A,B,C,D}q^{[M,N]-[A,C]-[B,D]-\langle
A,D\rangle}F^{M}_{AB}F^{N}_{CD}\varepsilon_{AC}^{X}\varepsilon_{BD}^{Y}.$$
Then
\begin{eqnarray}
   && \sum_{E}\varepsilon_{MN}^{E}X_E  \nonumber\\
   &=& \sum_{E,X,Y}\varepsilon_{MN}^{E}q^{-\frac{1}{2}\langle Y,X\rangle}F^{E}_{XY}X^{\widetilde{B}\underline{y}-(\widetilde{I}-\widetilde{R})\underline{e}}\nonumber\\
  &=&\sum_{A,B,C,D,X,Y}q^{[M,N]-[A,C]-[B,D]-\langle
A,D\rangle-\frac{1}{2}\langle
B+D,A+C\rangle}F^{M}_{AB}F^{N}_{CD}\varepsilon_{AC}^{X}\varepsilon_{BD}^{Y}X^{\widetilde{B}\underline{y}-(\widetilde{I}-\widetilde{R})\underline{e}}.\nonumber
\end{eqnarray}
Since
\begin{eqnarray}
   && X^{\widetilde{B}\underline{y}-(\widetilde{I}-\widetilde{R})\underline{e}} \nonumber\\
   &=& X^{\widetilde{B}(\underline{b}+\underline{d})-(\widetilde{I}-\widetilde{R})(\underline{m}+\underline{n})}\nonumber\\
   &=& q^{-\frac{1}{2}\Lambda(\widetilde{B}\underline{d}-(\widetilde{I}-\widetilde{R})\underline{n},
\widetilde{B}\underline{b}-(\widetilde{I}-\widetilde{R})\underline{m})}X^{\widetilde{B}\underline{d}-(\widetilde{I}-\widetilde{R})\underline{n}}
X^{\widetilde{B}\underline{b}-(\widetilde{I}-\widetilde{R})\underline{m}}\nonumber\\
   &=& q^{-\frac{1}{2}\Lambda((\widetilde{I}-\widetilde{R})\underline{n},
(\widetilde{I}-\widetilde{R})\underline{m})-\frac{1}{2}[\langle
D,B\rangle-\langle B,D\rangle-\langle D,M\rangle+\langle
B,N\rangle]}X^{\widetilde{B}\underline{d}-(\widetilde{I}-\widetilde{R})\underline{n}}
X^{\widetilde{B}\underline{b}-(\widetilde{I}-\widetilde{R})\underline{m}}\nonumber\\
  &=&q^{-\frac{1}{2}\Lambda((\widetilde{I}-\widetilde{R})\underline{n},
(\widetilde{I}-\widetilde{R})\underline{m})}q^{\frac{1}{2}\langle
D,A\rangle-\frac{1}{2}\langle
B,C\rangle}X^{\widetilde{B}\underline{d}-(\widetilde{I}-\widetilde{R})\underline{n}}
X^{\widetilde{B}\underline{b}-(\widetilde{I}-\widetilde{R})\underline{m}}.\nonumber
\end{eqnarray}
Thus
\begin{eqnarray}
   && \sum_{E}\varepsilon_{MN}^{E}X_E  \nonumber\\
   &=& q^{\frac{1}{2}\Lambda((\widetilde{I}-\widetilde{R})\underline{m},
(\widetilde{I}-\widetilde{R})\underline{n})}\sum_{A,B,C,D}q^{[M,N]-[A,C]-[B,D]-\langle
A,D\rangle-\frac{1}{2}\langle
B+D,A+C\rangle+[A,C]^{1}+[B,D]^{1}}\cdot\nonumber\\
  &&q^{\frac{1}{2}\langle
D,A\rangle-\frac{1}{2}\langle
B,C\rangle}F^{M}_{AB}F^{N}_{CD}X^{\widetilde{B}\underline{d}-(\widetilde{I}-\widetilde{R})\underline{n}}
X^{\widetilde{B}\underline{b}-(\widetilde{I}-\widetilde{R})\underline{m}}.\nonumber
\end{eqnarray}
Here we use the following fact
$$\sum_{X}\varepsilon_{AC}^{X}=q^{[A,C]^{1}},\sum_{Y}\varepsilon_{BD}^{Y}=q^{[B,D]^{1}}$$
Note that
$$[M,N]-[A,C]-[B,D]-\langle A,D\rangle+[A,C]^{1}+[B,D]^{1}=[M,N]^{1}+\langle B, C\rangle.$$
Hence
\begin{eqnarray}
   && \sum_{E}\varepsilon_{MN}^{E}X_E  \nonumber\\
   &=& q^{\frac{1}{2}\Lambda((\widetilde{I}-\widetilde{R})\underline{m},
(\widetilde{I}-\widetilde{R})\underline{n})}q^{[M,N]^{1}}\sum_{A,B,C,D}q^{\langle
B,C\rangle-\frac{1}{2}\langle B,C\rangle-\frac{1}{2}\langle
D,A\rangle+\frac{1}{2}\langle D,A\rangle-\frac{1}{2}\langle
B,C\rangle}\cdot\nonumber\\
  && F^{N}_{CD}q^{-\frac{1}{2}\langle D,C\rangle}X^{\widetilde{B}\underline{d}-(\widetilde{I}-\widetilde{R})\underline{n}}
F^{M}_{AB}q^{-\frac{1}{2}\langle
B,A\rangle}X^{\widetilde{B}\underline{b}-(\widetilde{I}-\widetilde{R})\underline{m}}
\nonumber\\
 &=&q^{\frac{1}{2}\Lambda((\widetilde{I}-\widetilde{R})\underline{m},
(\widetilde{I}-\widetilde{R})\underline{n})}q^{[M,N]^{1}}X_{N}X_{M}.\nonumber
\end{eqnarray}
This completes the proof.
\end{proof}

\begin{remark}
Theorem \ref{hall multi} is similar to the multiplication formula in
dual Hall algebras. It is reasonable to conjecture that it provides
some PBW-type basis (\cite{GP}) in the corresponding quantum cluster
algebra.
\end{remark}

Let $M,N$ be $kQ-$modules and  assume that
$$\mathrm{dim}_{k}\mathrm{Ext}^{1}_{k\widetilde{Q}}(M,N)=\mathrm{dim}_{k}\mathrm{Hom}_{k\widetilde{Q}}(N,\tau M)=1.$$
Then there are two ``canonical'' exact
sequences
$$\varepsilon:\quad 0\longrightarrow N\longrightarrow E\longrightarrow M\longrightarrow 0$$
$$\varepsilon': \quad 0\longrightarrow D_{0}\longrightarrow N\longrightarrow \tau M\longrightarrow \tau A\oplus I\longrightarrow 0$$
which induces the $k$-bases of
$\mathrm{Ext}^{1}_{k\widetilde{Q}}(M,N)$ and
$\mathrm{Hom}_{k\widetilde{Q}}(N,\tau M)$, respectively. We fix
them. Set $M=M'\oplus P_0, A_0=A\oplus P_0$ where $P_0$ is a
projective $k\widetilde{Q}$-module, $A$ and $M'$ have no projective
summands. The exact sequences also provide the two non-split
triangles in $\mathcal{C}_{\widetilde{Q}}$:
$$N\longrightarrow E\longrightarrow M\longrightarrow N[1]=\tau N$$
and
$$M\longrightarrow D_{0}\oplus  A_0\oplus I[-1]\longrightarrow N\longrightarrow \tau M.$$

Now we state the first part of our multiplication theorem for
acyclic quantum cluster algebras, which can be viewed as a quantum
analogue of the one-dimensional Caldero-Keller multiplication
theorem in \cite{CK2}. The main idea in the proof comes from
\cite{Hubery}.
\begin{theorem}\label{multi-formula}
With the above notation,  assume that
$\mathrm{Hom}_{k\widetilde{Q}}(D_0,\tau A_0\oplus
I)=\mathrm{Hom}_{k\widetilde{Q}}(A_0,I)=0.$ Then the following
formula holds
$$ X_{N}X_M=q^{\frac{1}{2}\Lambda((\widetilde{I}-\widetilde{R})\underline{n},
(\widetilde{I}-\widetilde{R})\underline{m})}X_E+q^{\frac{1}{2}\Lambda((\widetilde{I}-\widetilde{R})\underline{n},
(\widetilde{I}-\widetilde{R})\underline{m})+\frac{1}{2}\langle
M,N\rangle-\frac{1}{2}\langle A_0, D_0\rangle}X_{D_0\oplus A_0\oplus
I[-1]}.$$
\end{theorem}
Here, we note that
$$\frac{q^{{[M,N]^1}}-1}{q-1}X_{N}X_M=X_{N}X_M.$$
\begin{proof}
By definition, we have
\begin{eqnarray}
   && X_{N}X_M  \nonumber\\
   &=& \sum_{C,D}q^{-\frac{1}{2}\langle D,C\rangle}F^{N}_{CD}X^{\widetilde{B}\underline{d}-(\widetilde{I}-\widetilde{R})\underline{n}}
   \sum_{A,B}q^{-\frac{1}{2}\langle B,A\rangle}F^{M}_{AB}X^{\widetilde{B}\underline{b}-(\widetilde{I}-\widetilde{R})\underline{m}}\nonumber\\
  &=& \sum_{A,B,C,D}F^{M}_{AB}F^{N}_{CD}q^{-\frac{1}{2}\langle D,C\rangle-\frac{1}{2}\langle B,A\rangle+\frac{1}{2}\Lambda(\widetilde{B}\underline{d}-
  (\widetilde{I}-\widetilde{R})\underline{n},
\widetilde{B}\underline{b}-(\widetilde{I}-\widetilde{R})\underline{m})}X^{\widetilde{B}(\underline{b}+\underline{d})-(\widetilde{I}-\widetilde{R})(\underline{m}+\underline{n})}
\nonumber\\
 &=&\sum_{A,B,C,D}F^{M}_{AB}F^{N}_{CD}q^{-\frac{1}{2}\langle B+D,A+C\rangle}q^{\langle B,C\rangle}q^{\frac{1}{2}\Lambda((\widetilde{I}-\widetilde{R})\underline{n},
(\widetilde{I}-\widetilde{R})\underline{m})}X^{\widetilde{B}(\underline{b}+\underline{d})-(\widetilde{I}-\widetilde{R})(\underline{m}+\underline{n})}.\nonumber
\end{eqnarray}

We set
$$s_1:=\sum_{E\ncong M\oplus N}\frac{\varepsilon_{MN}^{E}}{q-1}X_E=\sum_{X,Y,E\ncong M\oplus N}
\frac{\varepsilon_{MN}^{E}}{q-1}F^{E}_{XY}q^{-\frac{1}{2}\langle Y,X\rangle}
   X^{\widetilde{B}\underline{y}-(\widetilde{I}-\widetilde{R})\underline{e}}$$
As in the proof of Theorem \ref{hall multi}, we have
\begin{eqnarray}
   && \sum_{X,Y,E}\varepsilon_{MN}^{E}X_E  \nonumber\\
   &=& \sum_{A,B,C,D,X,Y}q^{[M,N]-[A,C]-[B,D]-\langle
A,D\rangle-\frac{1}{2}\langle
B+D,A+C\rangle}F^{M}_{AB}F^{N}_{CD}\varepsilon_{AC}^{X}\varepsilon_{BD}^{Y}X^{\widetilde{B}\underline{y}-(\widetilde{I}-\widetilde{R})\underline{e}}\nonumber\\
  &=& \sum_{A,B,C,D}q^{[M,N]^{1}+\langle
B,C\rangle-\frac{1}{2}\langle
B+D,A+C\rangle}F^{M}_{AB}F^{N}_{CD}X^{\widetilde{B}\underline{y}-(\widetilde{I}-\widetilde{R})\underline{e}}.\nonumber
\end{eqnarray}
On the other hand
$$X_{M\oplus N}=\sum_{A,B,C,D}q^{[B,C]-\frac{1}{2}\langle
B+D,A+C\rangle}F^{M}_{AB}F^{N}_{CD}X^{\widetilde{B}(\underline{b}+\underline{d})-(\widetilde{I}-\widetilde{R})(\underline{m}+\underline{n})}.$$
Thus
$$s_1=\sum_{A,B,C,D}\frac{q^{[M,N]^{1}}-q^{[B,C]^{1}}}{q-1}q^{\langle B,C\rangle-\frac{1}{2}\langle
B+D,A+C\rangle}F^{M}_{AB}F^{N}_{CD}X^{\widetilde{B}(\underline{b}+\underline{d})-(\widetilde{I}-\widetilde{R})(\underline{m}+\underline{n})}.$$

Thirdly we compute the term
$$s_2:=\sum_{A,D_0,I,D_0\ncong N}\frac{|\mathrm{Hom}_{k\widetilde{Q}}(N,\tau M)_{D_{0}AI}|}{q-1}X_{A_{0}\oplus D_{0}\oplus I[-1]}.$$
Here, we use the following notation as in \cite{Hubery}
$$\mathrm{Hom}_{k\widetilde{Q}}(N,\tau M)_{D_0AI}:=\{f\neq 0: N\longrightarrow \tau
M|\mathrm{ker}f\cong D_0, \mathrm{coker}f\cong \tau A\oplus I\}.$$
Note that $\mathrm{dim}_{k}\mathrm{Hom}_{k\widetilde{Q}}(N,\tau
M)=1,$ we have the following exact sequences
$$0\longrightarrow B_0\longrightarrow M\longrightarrow A_0\longrightarrow 0$$
$$0\longrightarrow C\longrightarrow \tau B_0\longrightarrow I\longrightarrow 0$$
where $C=\mathrm{im}f, \mathrm{ker}f=D_0.$

\begin{eqnarray}
   s_2&=& \frac{|\mathrm{Hom}_{k\widetilde{Q}}(N,\tau M)|-1}{q-1}X_{A_0\oplus D_0\oplus I[-1]}  \nonumber\\
   &=& \sum_{X,Y,K,L}\frac{|\mathrm{Hom}_{k\widetilde{Q}}(N,\tau M)|-1}{q-1}F^{D_0}_{XY}F^{A_0}_{KL}q^{[L,X]-\frac{1}{2}\langle
Y+L,K+X\rangle}X^{\widetilde{B}(\underline{y}+\underline{l}+\underline{b_0})-(\widetilde{I}-\widetilde{R})(\underline{m}+\underline{n})}\nonumber\\
  &=&\sum_{A,B,C,D} \frac{q^{[C,\tau B]}-1}{q-1}F^{M}_{AB}F^{N}_{CD}q^{[L,X]-\frac{1}{2}\langle
Y+L,K+X\rangle}X^{\widetilde{B}(\underline{y}+\underline{l}+\underline{b_0})-(\widetilde{I}-\widetilde{R})(\underline{m}+\underline{n})}.\nonumber
\end{eqnarray}
Here, $Y=D,K=A,B=B_0+L$ in the above expression and the equality can
be illustrated by the following diagram:

$$
\xymatrix{&Y\ar@{=}[r]\ar[d]&Y\ar[d]&\tau A\ar@{=}[r]&\tau A&\\
 0\ar[r]&D_0\ar[r]\ar[d]&N\ar[r]\ar[d]&\tau
M\ar[r]\ar[u]&\tau A_0\oplus I\ar[r]\ar[u]&0\\
0\ar[r]&X\ar[r]&C\ar[r]&\tau B\ar[r]\ar[u]&\tau L\oplus
I\ar[r]\ar[u]&0}
$$
We must to check the relation between
$$-\frac{1}{2}\langle Y+L,K+X\rangle+[L,X]$$ and
$$-\frac{1}{2}\langle B+D,A+C\rangle+\langle B,C\rangle.$$
In this case, note that $ D=Y, L=A_0-A,K=A,[L,X]^{1}=[X,\tau L]=0.$
We have
\begin{eqnarray}
  -\frac{1}{2}\langle Y+L,K+X\rangle+[L,X] &=& -\frac{1}{2}\langle Y+L,K+X\rangle+\langle L,X\rangle \nonumber\\
   &=& -\frac{1}{2}\langle D+A_0-A,A+D_0-D\rangle+\langle A-A_0,D_0-D\rangle\nonumber\\
   &=& -\frac{1}{2}\langle D,A\rangle-\frac{1}{2}\langle D,D_0\rangle+\frac{1}{2}\langle D,D\rangle-\frac{1}{2}\langle A_0,A\rangle
   +\frac{1}{2}\langle A_0,D_0\rangle\nonumber\\
   && -\frac{1}{2}\langle A_0,D\rangle+\frac{1}{2}\langle A,A\rangle
   -\frac{1}{2}\langle A,D_0\rangle+\frac{1}{2}\langle A,D\rangle. \nonumber\
\end{eqnarray}
And
\begin{eqnarray}
   && -\frac{1}{2}\langle B+D,A+C\rangle+\langle B,C\rangle \nonumber\\
   &=& -\frac{1}{2}\langle M-A+D,A+N-D\rangle+\langle M-A,N-D\rangle\nonumber\\
   &=& -\frac{1}{2}\langle M,A\rangle-\frac{1}{2}\langle M,D\rangle+\frac{1}{2}\langle A,A\rangle-\frac{1}{2}\langle A,N\rangle+\frac{1}{2}\langle A,D\rangle\nonumber\\
   && -\frac{1}{2}\langle D,A\rangle-\frac{1}{2}\langle D,N\rangle+\frac{1}{2}\langle D,D\rangle+\frac{1}{2}\langle M,N\rangle. \nonumber\
\end{eqnarray}
Hence it is equivalent to compare
$$-\frac{1}{2}\langle D,D_0\rangle-\frac{1}{2}\langle A_0,A\rangle+\frac{1}{2}\langle A_0,D_0\rangle-\frac{1}{2}\langle A_0,D\rangle-\frac{1}{2}\langle A,D_0\rangle$$
and
$$-\frac{1}{2}\langle D,N\rangle-\frac{1}{2}\langle M,A\rangle+\frac{1}{2}\langle M,N\rangle-\frac{1}{2}\langle M,D\rangle-\frac{1}{2}\langle A,N\rangle.$$
We claim that
$$\langle D,N\rangle+\langle M,D\rangle=\langle D,D_0\rangle+\langle A_0,D\rangle$$
and
$$\langle A_0,A\rangle+\langle A,D_0\rangle=\langle M,A\rangle+\langle A,N\rangle.$$
Indeed, we have
\begin{eqnarray}
  \langle D,N-D_0\rangle &=& \langle D,\tau M-\tau A_0-I\rangle \nonumber\\
   &=& \langle D,\tau M-\tau A_0\rangle=\langle A_0-M,D\rangle. \nonumber\
\end{eqnarray}
In the same way, we also have
$$
\langle A, N-D_0\rangle=\langle A_0-M, A\rangle.
$$

 Thus
$$s_2=q^{\frac{1}{2}\langle
A_0,D_0\rangle-\frac{1}{2}\langle
M,N\rangle}\sum_{A,B,C,D}\frac{q^{[B,C]^{1}}-1}{q-1}q^{\langle
B,C\rangle-\frac{1}{2}\langle
B+D,A+C\rangle}F^{M}_{AB}F^{N}_{CD}X^{\widetilde{B}(\underline{b}+\underline{d})-(\widetilde{I}-\widetilde{R})(\underline{m}+\underline{n})}.$$

Therefore we have the following multiplication formula
$$ X_{N}X_M=q^{\frac{1}{2}\Lambda((\widetilde{I}-\widetilde{R})\underline{n},
(\widetilde{I}-\widetilde{R})\underline{m})}X_E+q^{\frac{1}{2}\Lambda((\widetilde{I}-\widetilde{R})\underline{n},
(\widetilde{I}-\widetilde{R})\underline{m})+\frac{1}{2}\langle
M,N\rangle-\frac{1}{2}\langle A_0,D_0\rangle}X_{D_0\oplus A_0\oplus
I[-1]}.$$
\end{proof}

There are three canonical special cases satisfying the assumption
$\mathrm{Hom}_{k\widetilde{Q}}(D_0,\tau A\oplus
I)=\mathrm{Hom}_{k\widetilde{Q}}(A_0,I)=0$ in Theorem
\ref{multi-formula}.

\noindent\textbf{Special case I}. Assume that $A_0=0=I.$ Then
$L=K=0=A.$ If $B\neq M,$ i.e, $B\subsetneqq M,$ then there exists
$f_1: N\longrightarrow \tau M$ induced by the above diagram which is
not surjective. It is a contradiction to the assumption
$\mathrm{dim}_{k}\mathrm{Hom}_{k\widetilde{Q}}(N,\tau M)=1.$ In this
case, the multiplication formula is
$$ X_{N}X_M=q^{\frac{1}{2}\Lambda((\widetilde{I}-\widetilde{R})\underline{n},
(\widetilde{I}-\widetilde{R})\underline{m})}X_E+q^{\frac{1}{2}\Lambda((\widetilde{I}-\widetilde{R})\underline{n},
(\widetilde{I}-\widetilde{R})\underline{m})+\frac{1}{2}\langle
M,N\rangle}X_{D_0}.$$

\noindent\textbf{Special case II}. Assume that $D_0=0$ and
$\mathrm{Hom}_{k\widetilde{Q}}(A_0,I)=0.$ Then $Y=X=0, C=N.$ In this
case, the multiplication formula is
$$ X_{N}X_M=q^{\frac{1}{2}\Lambda((\widetilde{I}-\widetilde{R})\underline{n},
(\widetilde{I}-\widetilde{R})\underline{m})}X_E+q^{\frac{1}{2}\Lambda((\widetilde{I}-\widetilde{R})\underline{n},
(\widetilde{I}-\widetilde{R})\underline{m})+\frac{1}{2}\langle
M,N\rangle}X_{A_0\oplus I[-1]}.$$

\noindent\textbf{Special case III}. Assume that $M,N$ are
indecomposable rigid $kQ$-mod and
$$\mathrm{dim}_k\mathrm{Ext}^1_{\mathcal{C}_{\widetilde{Q}}}(M, N)=1.$$
Since $D_0\oplus A_0\oplus I[-1]$ is rigid, then the assumption
$\mathrm{Hom}_{k\widetilde{Q}}(D_0,\tau A\oplus
I)=\mathrm{Hom}_{k\widetilde{Q}}(A_0,I)=0$ in Theorem
\ref{multi-formula} holds.
\begin{lemma}\label{special}
With the assumption in Special case III, we have $\frac{1}{2}\langle
A_0,D_0\rangle-\frac{1}{2}\langle M,N\rangle=\frac{1}{2}.$
\end{lemma}
\begin{proof}
Note that we have $$\frac{1}{2}\langle
A_0,D_0\rangle-\frac{1}{2}\langle M,N\rangle=\frac{1}{2}\langle
A_0,N-N/D_0\rangle-\frac{1}{2}\langle M,N\rangle.$$ We need to
confirm the two equations
\begin{enumerate}
  \item $\langle
M,N\rangle=\langle A_0,N\rangle-1$ and
  \item $\langle
A_0,N/D_0\rangle=0.$
\end{enumerate}
Note that $A_0\oplus N$ is rigid, thus $[A_0,N]^{1}=0.$ We have the
following exact sequences
$$0\longrightarrow N/D_0\longrightarrow \tau M\longrightarrow \tau A_0 \oplus I\longrightarrow 0$$
$$0\longrightarrow D_0\longrightarrow N\longrightarrow N/D_0\longrightarrow 0$$
Applying the functor $\mathrm{Hom}_{k\widetilde{Q}}(N,-)$, we have
the following exact sequences
$$[N,N/D_0]^{1}\longrightarrow [N,\tau M]^{1}\longrightarrow [N,\tau A_0\oplus I]^{1}\longrightarrow 0$$
$$[N,N]^{1}\longrightarrow [N,N/D_0]^{1}\longrightarrow 0.$$
Hence
$$\langle
M,N\rangle=[M,N]-1=[A_0,N]-1=\langle A_0,N\rangle-1.$$ As for the
second equation, apply the functor
$\mathrm{Hom}_{k\widetilde{Q}}(A_0,-)$ to the exact sequence
$$0\longrightarrow D_0\longrightarrow N\longrightarrow N/D_0\longrightarrow 0$$
We have the following exact sequence
$$[A_0,N]^{1}\longrightarrow [A_0,N/D_0]^{1}\longrightarrow 0$$
Thus $[A_0,N/D_0]^{1}=0.$ Applying the functor
$\mathrm{Hom}_{k\widetilde{Q}}(\tau M,-)$ to the exact sequence
$$0\longrightarrow N/D_0\longrightarrow \tau M\longrightarrow \tau A_0 \oplus I\longrightarrow 0$$
We have the following exact sequence
$$[\tau M,\tau M]^{1}\longrightarrow [\tau M,\tau A_0\oplus I]^{1}\longrightarrow 0$$
Thus we have $$[M,A_0]^{1}=[A_0,\tau M]=0.$$ Again applying the
functor $\mathrm{Hom}_{k\widetilde{Q}}(A_0,-)$, we have the exact
sequence
$$0\longrightarrow [A_0,N/D_0]\longrightarrow [A_0,\tau M]=0$$
Hence $[A_0,N/D_0]=0.$
\end{proof}

By Lemma \ref{special}, we obtain the following multiplication
theorem between quantum cluster variables in \cite{fanqin}.
\begin{corollary}
Let $M$ and $N$ be indecomposable rigid $kQ$-modules and
$\mathrm{dim}_k\mathrm{Ext}^1_{\mathcal{C}_{\widetilde{Q}}}(M,
N)=1.$ Let
$$N\longrightarrow E\longrightarrow M\longrightarrow N[1]=\tau N$$
and
$$M\longrightarrow D_{0}\oplus A_{0}\oplus I[-1]\longrightarrow N\longrightarrow \tau M$$
be two non-split triangles in $\mathcal{C}_{\widetilde{Q}}$ as
above. Then we have
$$ X_{N}X_M=q^{\frac{1}{2}\Lambda((\widetilde{I}-\widetilde{R})\underline{n},
(\widetilde{I}-\widetilde{R})\underline{m})}X_E+q^{\frac{1}{2}\Lambda((\widetilde{I}-\widetilde{R})\underline{n},
(\widetilde{I}-\widetilde{R})\underline{m})-\frac{1}{2}}X_{D_0\oplus
A_0\oplus I[-1]}.$$
\end{corollary}

Now let $M$ be a $kQ$-module and $P$ be a projective
$k\widetilde{Q}$-module with $[P,M]=[M,I]=1,$ where $I=\nu(P)$, here
$\nu=\mathrm{DHom}_{k\widetilde{Q}}(-,k\widetilde{Q})$ is the
Nakayama functor. It is well-known that $I$ is an injective module
with $\mathrm{soc}I=P/\mathrm{rad}P.$ Fix two nonzero morphisms
$f\in \mathrm{Hom}_{k\widetilde{Q}}(P, M)$ and $g\in
\mathrm{Hom}_{k\widetilde{Q}}(M, I).$ The two morphisms induce  the
following exact sequences
$$\xymatrix{0\ar[r]& P'\ar[r]& P\ar[r]^{f}& M\ar[r]& A\ar[r]& 0}$$
and
$$\xymatrix{0\ar[r]& B\ar[r]& M\ar[r]^{g}& I\ar[r]& I'\ar[r]& 0}.$$
These correspond to two non-split triangles in
$\mathcal{C}_{\widetilde{Q}}$
$$M\longrightarrow E'\longrightarrow P[1]\longrightarrow M[1]$$
and
$$I[-1]\longrightarrow E\longrightarrow M\longrightarrow I,$$
respectively, where $E\simeq B\oplus I'[-1]$ and $E'\simeq A\oplus
P'[1].$

Now we state the second part of our multiplication theorem for
acyclic quantum cluster algebras.
\begin{theorem}\label{exchange2}
With the above notations, assume that $[B,I']=[P',A]=0.$ Then we
have
$$X_{\tau P}X_{M}=q^{\frac{1}{2}\Lambda(\underline{\mathrm{dim}}P/rad P,
-(\widetilde{I}-\widetilde{R})\underline{m})}X_E+q^{\frac{1}{2}\Lambda(\underline{\mathrm{dim}}P/rad
P,
-(\widetilde{I}-\widetilde{R})\underline{m})-\frac{1}{2}}X_{E'}.$$
\end{theorem}
\begin{proof}
\begin{eqnarray}
   && X_{\tau P}X_{M}  \nonumber\\
   &=& X^{\underline{\mathrm{dim}} P/rad P}\sum_{G,H}q^{-\frac{1}{2}\langle H,G\rangle}F^{M}_{GH}X^{\widetilde{B}\underline{h}-(\widetilde{I}-\widetilde{R})\underline{m}}
  \nonumber\\
  &=& \sum_{G,H}q^{-\frac{1}{2}\langle H,G\rangle}F^{M}_{GH}q^{\frac{1}{2}\Lambda(\underline{\mathrm{dim}} P/\mathrm{rad}P,\widetilde{B}\underline{h}-(\widetilde{I}-\widetilde{R})
  \underline{m})}X^{\widetilde{B}\underline{h}-(\widetilde{I}-\widetilde{R})\underline{m}+\underline{\mathrm{dim}}
P/rad P}
\nonumber\\
 &=&q^{\frac{1}{2}\Lambda(\underline{\mathrm{dim}} P/\mathrm{rad}P,-(\widetilde{I}-\widetilde{R})\underline{m})}\sum_{G,H}q^{-\frac{1}{2}\langle
H,G\rangle}q^{-\frac{1}{2}[P,H]}F^{M}_{GH}X^{\widetilde{B}\underline{h}-(\widetilde{I}-\widetilde{R})\underline{m}+\underline{\mathrm{\mathrm{dim}}}
P/\mathrm{rad}P}.\nonumber
\end{eqnarray}
Here we use the following fact
$$\Lambda(\underline{\mathrm{dim}} P/\mathrm{rad}P,\widetilde{B}\underline{h})=(\underline{\mathrm{dim}} P/\mathrm{rad}P)^{tr}\Lambda\widetilde{B}\underline{h}=-
(\underline{\mathrm{dim}} P/\mathrm{rad}P)^{tr}\begin{bmatrix}I_n\\0
\end{bmatrix}\underline{h}=-[P,H].$$ Note that by assumption
$[P,M]=1,$ we have that $[P,H]=0\ or \ 1.$

 We firstly compute the  term $$X_E=X_{B\oplus
I'[-1]}=\sum_{X,Y}q^{-\frac{1}{2}\langle
Y,X\rangle}F^{B}_{XY}X^{\widetilde{B}\underline{y}-(\widetilde{I}-\widetilde{R})\underline{b}+\underline{\mathrm{dim}}
\mathrm{soc}I'}.$$  We have the following diagram
$$
\xymatrix{&0\ar[d]&0\ar[d]\\
&Y\ar@{=}[r]\ar[d]&Y\ar[d]\\
 0\ar[r]&B\ar[r]\ar[d]&M\ar[r]^{\theta}\ar[d]&I\ar[r]&I'\ar[r]&0\\
0\ar[r]&X\ar[r]\ar[d]&G\ar[d]\\
&0&0}
$$
and a short exact sequence
$$0\longrightarrow \mathrm{im}\theta \longrightarrow I\longrightarrow I'\longrightarrow 0.$$
As we assume that $[B,I']=0,$ thus $[H,I']=0.$ Then
\begin{eqnarray}
   && \langle Y,X\rangle-\langle H,G\rangle = \langle H,X\rangle-\langle H,G\rangle =\langle H,X-G\rangle =\langle H,B-M\rangle=-\langle H, \mathrm{im}\theta\rangle.\nonumber
\end{eqnarray}
Applying the functor $[H,-]$ to the above short exact sequence, we
have
$$0\longrightarrow [H, \mathrm{im}\theta]\longrightarrow [H, I]\longrightarrow [H, I']\longrightarrow [H, \mathrm{im}\theta]^{1}\longrightarrow 0$$
Note that $[H, I]=[H, I']=0,$ thus $\langle H,
\mathrm{im}\theta\rangle=0.$ Hence
$$X_E=\sum_{G,H,[P,H]=0}q^{-\frac{1}{2}\langle
H,G\rangle}F^{M}_{GH}X^{\widetilde{B}\underline{h}-(\widetilde{I}-\widetilde{R})\underline{m}+\underline{\mathrm{dim}}
P/\mathrm{rad}P}.$$

 Now compute the  term $$X_{E'}=X_{A\oplus
P'[1]}=\sum_{X,Y}q^{-\frac{1}{2}\langle
Y,X\rangle}F^{A}_{XY}X^{\widetilde{B}\underline{y}-(\widetilde{I}-\widetilde{R})\underline{a}+\underline{\mathrm{dim}}
P'/\mathrm{rad}P'}.$$  We have the following diagram
$$
\xymatrix{&&&0\ar[d]&0\ar[d]\\
&&P\ar[r]\ar@{=}[d]&H\ar[r]\ar[d]&Y\ar[r]\ar[d]&0\\
 0\ar[r]&P'\ar[r]&P\ar[r]^{\theta'}&M\ar[r]\ar[d]&A\ar[r]\ar[d]&0\\
&&&G\ar@{=}[r]\ar[d]&X\ar[d]\\
&&&0&0}
$$
Applying the functor $[P',-]$ to the following short exact sequence
$$0\longrightarrow Y \longrightarrow A\longrightarrow G\longrightarrow 0.$$
 we
have
$$0\longrightarrow [P',Y]\longrightarrow [P', A]\longrightarrow [P', G]\longrightarrow 0$$
As we assume that $[P',A]=0,$ thus $[P',G]=0.$ Then $\langle
P',G\rangle=0.$ Hence we have
\begin{eqnarray}
   && \langle Y,X\rangle-\langle H,G\rangle = \langle Y,G\rangle-\langle H,G\rangle =\langle Y-H,G\rangle =\langle
   A-M,G\rangle\nonumber\\
   &=& \langle P'-P,G\rangle=\langle P',G\rangle=0.\nonumber
\end{eqnarray}
Therefore $$X_E'=\sum_{G,H,[P,H]=1}q^{-\frac{1}{2}\langle
H,G\rangle}F^{M}_{GH}X^{\widetilde{B}\underline{h}-(\widetilde{I}-\widetilde{R})\underline{m}+\underline{\mathrm{dim}}
P/\mathrm{rad}P}.$$ This completes the proof.
\end{proof}

Note that if $M$ is indecomposable rigid object in
$\mathcal{C}_{\widetilde{Q}}$ and $[P,M]=[M,I]=1$, then both
$E=B\oplus I'[-1]$ and $E'=A\oplus P'[1]$ are rigid. Thus the
assumptions
$\mathrm{Hom}_{k\widetilde{Q}}(B,I')=\mathrm{Hom}_{k\widetilde{Q}}(P',A)=0$
in Theorem \ref{exchange2} naturally hold. The quantum cluster
multiplication theorem in \cite{fanqin} deals with this special
case.

\section{$\mathbb{ZP}$-bases in  specialized quantum cluster algebras of finite type}
Let $k$ be a finite field with cardinality $|k|=q$ and $m\geq n$ be
two positive integers and $\widetilde{Q}$ an acyclic quiver with
vertex set $\{1,\ldots,m\}$. The full subquiver $Q$ on the vertices
$\{1,\ldots,n\}$ is  the principal part of $\widetilde{Q}$.
 Let $\mathcal{A}_{|k|}(\widetilde{Q})$ be the corresponding specialized
quantum cluster algebra of $Q$ with coefficients. Then the main
theorem in \cite{fanqin} shows that
$\mathcal{A}_{|k|}(\widetilde{Q})$ is the $\mathbb{ZP}$-subalgebra
of $\Fcal$ generated by
$$\{X_M| M\ \text{is  indecomposable rigid $kQ$-mod}\}\cup$$$$
\{X_{\tau P_i},1\leq i\leq n|P_i\ \text{is  indecomposable
projective $k\widetilde{Q}$-mod}\}.$$ Let $i$ be a sink or a source
in $\widetilde{Q}$. We define the reflected quiver
$\sigma_i(\widetilde{Q})$ by reversing all the arrows ending at $i$.
An \emph{admissible sequence of sinks (resp. sources)} is a sequence
$(i_1, \ldots, i_l)$ such that $i_1$ is a sink (resp. source) in
$\widetilde{Q}$ and $i_k$ is a sink (resp source) in
$\sigma_{i_{k-1}}\cdots \sigma_{i_1}(\widetilde{Q})$ for any $k=2,
\ldots, l$. A quiver $\widetilde{Q}'$ is called
\emph{reflection-equivalent}\index{reflection-equivalent} to
$\widetilde{Q}$ if there exists an admissible sequence of sinks or
sources $(i_1, \ldots, i_l)$ such that
$\widetilde{Q}'=\sigma_{i_{l}}\cdots \sigma_{i_1}(\widetilde{Q})$.
Note that mutations can be viewed as generalizations of reflections,
i.e, if $i$ is a sink or a source in a quiver $\widetilde{Q}$,
 then $\mu_i(\widetilde{Q})=\sigma_i(\widetilde{Q})$ where $\mu_i$ denotes the mutation in the direction
 $i$. Thus if $\widetilde{Q}'$ is a quiver mutation-equivalent to $\widetilde{Q}$,
 there is a natural canonical isomorphism between
$\mathcal{A}_{|k|}(\widetilde{Q})$ and
$\mathcal{A}_{|k|}(\widetilde{Q}'),$ denoted by
$$\Phi_{i}: \mathcal{A}_{|k|}(\widetilde{Q})\rightarrow
\mathcal{A}_{|k|}(\widetilde{Q}').$$

Let $\Sigma_i^+:\ \mathrm{mod}(\widetilde{Q}) \longrightarrow \
\mathrm{mod}(\widetilde{Q}')$ be the standard BGP-reflection functor
and $R_i^+:\mathcal C_{\widetilde{Q}} \longrightarrow \mathcal
C_{\widetilde{Q}'}$ be the
        extended BGP-reflection functor defined by \cite{Zhu}:
        $$R_i^+:\left\{\begin{array}{rcll}
            X & \mapsto & \Sigma_i^+(X), & \textrm{ if }X \not \simeq S_i \textrm{ is a $kQ$-module,}\\
            S_i & \mapsto & P_i[1], \\
            P_j[1] & \mapsto & P_j[1], & \textrm{ if }j \neq i,\\
            P_i[1] & \mapsto & S_i.
        \end{array}\right.$$
    By Rupel \cite{rupel}, the following holds:
\begin{theorem}\cite[Theorem 2.4, Lemma 5.6]{rupel}\label{ref}
For any $ X_M^{\widetilde{Q}}\in\mathcal{A}_{|k|}(\widetilde{Q})$,
we have $\Phi_{i}(X_M^{\widetilde{Q}})=X_{R_i^+M}^{\widetilde{Q}'}.$
\end{theorem}

\begin{definition}[\cite{CK1}]
            Let $Q$ be an acyclic quiver with associated matrix $B$. $Q$ will be called \emph{graded}\index{graded} if
            there exists a linear form $\epsilon$ on $\mathbb{Z}^{n}$ such that $\epsilon(B \alpha_i)<0$ for any $1\leq i \leq n$
            where $\alpha_i$ still denotes the $i$-th
            vector of the canonical basis of $\mathbb{Z}^{n}$.
\end{definition}

If $Q$ is a graded quiver, then it is proved in \cite{CK1} that we
can endow the cluster algebra of $Q$ with a grading. Namely, the
results are the following:

For any Laurent polynomial $P$ in the variables $X_i$, the
 $supp(P)$ of $P$ is defined as the set of points
$\lambda=(\lambda_i,1\leq i \leq n)$ of $\mathbb{Z}^{n}$ such that
the $\lambda$-component, that is, the coefficient of $\prod_{1\leq i
\leq n} X_i^{\lambda_i}$ in $P$ is nonzero. For any $\lambda$ in
$\mathbb{Z}^{n}$, let $C_\lambda$ be the convex cone with vertex
$\lambda$ and edge vectors generated by the $B\alpha_i$ for any
$1\leq i \leq n$. Then we have the following two propositions as the
quantum versions of Proposition 5 and Proposition 7 in \cite{CK1}
respectively.
\begin{proposition}\label{prop:supportconeCK1}
            Let $Q$ be a graded acyclic quiver with no multiple arrows and
           $M=M_0 \oplus P_M[1]$ with $M_0$ is $kQ$-module and $P_M$ projective $k\widetilde{Q}$-module. Then,
            $supp(X_{M_0\oplus P_M[1]})$
            is in $C_{\lambda_M}$ with
            $\lambda_M:=(-\langle\alpha_i,\underline{dim} M_0\rangle+\langle\underline{dim} P_M, \alpha_i\rangle)
            _{1\leq i \leq n}$.  Moreover, the $\lambda_M$-component of $X_{M_0\oplus P_M[1]}$
          is some nonzero monomials in $\{|k|^{\pm\frac{1}{2}},X^{\pm 1
            }_{n+1},\cdots,X^{\pm 1
            }_{m}\}$.
\end{proposition}

\begin{proposition}\label{prop:graduationCK1}
            Let $Q$ be a graded acyclic quiver with no multiple arrows. For any $m \in \mathbb{Z}$, set
            $$F_m=\left( \bigoplus_{\epsilon(\nu) \leq m} \mathbb{ZP}\prod_{1\leq i \leq n}u_i^{\nu_i}\right) \cap \Acal_{|k|}(\widetilde{Q}),$$
            then the set $(F_m)_{m \in \mathbb{Z}}$ defines a $\mathbb{Z}$-filtration of $\Acal_{|k|}(\widetilde{Q})$.
\end{proposition}

For any $\underline{d}\in \mathbb{Z}^{n},$ define
$\underline{d}^{+}=(d^+_i)_{1\leq i \leq n}$ such that $d^+_i=d_i$
if $d_i>0$ and $d_i^+=0$ if $d_i\leq 0$ for any $1\leq i \leq n.$
Dually, we set $\underline{d}^-=\underline{d}^+-\underline{d}.$ The
following proposition \ref{10} can be viewed as the categorification
of   \cite[Theorem 7.3]{berzel}.

\begin{proposition}\label{10}
Let $\widetilde{Q}$ be an acyclic quiver. Then the set
$\{\prod_{i=1}^{n}X^{d^+_i}_{S_i}X^{d^-_i}_{P_i[1]}\mid
(d_1,\cdots,d_n)\in \mathbb{Z}^n\}$ is a $\mathbb{ZP}$-basis of
$\Acal_{|k|}(\widetilde{Q})$.
\end{proposition}
\begin{proof}
For any $1\leq i\leq n,$ it is easy to check that the following set
is a cluster $$\{X_{\tau P_{1}},\cdots,X_{\tau P_{i-1}},
X_{S_{i}},X_{\tau P_{i+1}},\cdots, X_{\tau P_{n}}\}$$ obtained by
the mutation in direction $i$ of the cluster $$\{X_{\tau
P_{1}},\cdots,X_{\tau P_{i-1}}, X_{\tau P_{i}},X_{\tau
P_{i+1}},\cdots, X_{\tau P_{n}}\}.$$ Then the proposition
immediately follows from \cite[Theorem 7.3]{berzel} and
\cite[Theorem 5.4.3]{fanqin}.
\end{proof}

The main result is the following theorem showing the
$\mathbb{ZP}$-basis in a quantum cluster algebra of finite type. It
is the good bases in a cluster algebra of finite type in \cite{CK1}
by specializing $q$ and coefficients to $1$ and the existence of
Hall polynomials for representation direct algebras \cite{ringel:2}.

\begin{theorem}\label{12}
Let $Q$ is a simple-laced Dynkin quiver with $Q_0=\{1,2,\cdots,
n\}.$ Then the set $\mathcal{B}(Q):=\{X_{M}|M=M_0 \oplus P_M[1]\
\text{with}\ M_0\ \text{is}\   kQ\text{-module},\  P_M\ \
\text{projective}\ k\widetilde{Q}\text{-module},\
 M\ \text{rigid}$
$\text{object in}\ \mathcal C_{\widetilde{Q}}\}$ is a
$\mathbb{ZP}$-basis of $\Acal_{|k|}(\widetilde{Q})$.
\end{theorem}
\begin{proof}
It is obvious to see that there exists an orientation such that $Q'$
is a graded quiver where $Q'$ is reflection-equivalent to $Q$.
Assume that $\sigma_{i_{l}}\cdots \sigma_{i_1}(Q')=Q$. For any
$X_{M}\in \mathcal{B}(Q)$  with dimension vector
$\mathrm{\underline{dim}}M=\underline{m}=(m_1,\cdots,m_n)\in
\mathbb{Z}^n$, we know that $X_{M}\in \Acal_{|k|}(\widetilde{Q}')$.
Then by Proposition \ref{10} we have
$$X^{\widetilde{Q}'}_{M}=b_{\underline{m}}\prod_{i=1}^{n}(X_{S_i}^{\widetilde{Q}'})^{m^+_i}(X^{\widetilde{Q}'}_{P_i[1]})^{m^-_i}
+\sum_{\epsilon(\underline{l})<
\epsilon(\underline{m})}b_{\underline{l}}\prod_{i=1}^{n}(X_{S_i}^{\widetilde{Q}'})^{l^+_i}(X^{\widetilde{Q}'}_{P_i[1]})^{l^-_i}$$
where $\underline{l}=(l^+_i-l^-_i)_{i\in Q_0}$, $b_{\underline{m}}$
and $b_{\underline{l}}\in \mathbb{ZP}$. As $Q'$ is a graded quiver,
then by Proposition \ref{prop:supportconeCK1}, Proposition
\ref{prop:graduationCK1}, it follows that $b_{\underline{m}}$ must
be some nonzero monomial in $\{q^{\pm\frac{1}{2}},X^{\pm 1
            }_{n+1},\cdots,X^{\pm 1
            }_{m}\}$. Therefore,  $\mathcal{B}(Q)$ is a $\mathbb{ZP}$-basis of
$\Acal_{|k|}(\widetilde{Q}')$. There is a natural isomorphism:
$\Phi_{i_{l}}\cdots \Phi_{i_1}:
\mathcal{A}_{|k|}(\widetilde{Q}')\rightarrow
\mathcal{A}_{|k|}(\widetilde{Q})$. By Theorem \ref{ref}, we obtain
that
$$\Phi_{i_{l}}\cdots \Phi_{i_1}(X^{\widetilde{Q}'}_M)=X^{\widetilde{Q}}_{R^+_{i_{l}}\cdots
R^+_{i_1}(M)}.$$ Hence, $\mathcal{B}(Q)$ is a $\mathbb{ZP}$-basis of
$\Acal_{|k|}(\widetilde{Q})$.

\end{proof}
\section{$\mathbb{ZP}$-bases in   quantum cluster algebras of affine type}

\subsection{The case in the Kronecker quiver} Let $Q$ be the tame quiver of type
$\widetilde{A}^{(1)}_{1}$ as follows

\begin{center}
 \setlength{\unitlength}{0.61cm}
 \begin{picture}(5,4)
 \put(0,2){1}\put(0.4,2.2){$\bullet$}
\put(3,2.2){$\bullet$}\put(3.4,2){2} \put(0.8,2.5){\vector(3,0){2}}
\put(0.8,2.2){\vector(3,0){2}}
 %\put(1.7,2.7){$\alpha$}\put(1.7,1.5){$\beta$}
 \end{picture}
 \end{center}

It is well known that the  regular indecomposable modules decomposes
into a direct sum of homogeneous tubes  indexed by the projective
line $\mathbb{P}^1$. We denote the  regular indecomposable modules
in the homogeneous tube for $p\in \mathbb{P}^1$ of degree $1$ by
$R_p(n)$ where $n\in\mathbb{N}$ and
$\underline{\mathrm{dim}}R_p(n)=(n, n)$.

We consider the following  ice quiver $\widetilde{Q}$ with frozen
vertices $3$ and $4$:
$$
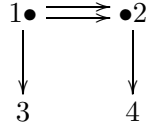
\xymatrix{1\bullet\ar[d]\ar @<2pt>[r] \ar@<-2pt>[r]& \bullet 2\ar[d]\\
3&  4}$$ Thus we have
$$\widetilde{R}=\left(\begin{array}{cc} 0 & 0\\
2& 0\\
1& 0\\
0 & 1\end{array}\right),\ \widetilde{I}=\left(\begin{array}{cc} 1 & 0\\
0& 1\\
0& 0\\
0 & 0\end{array}\right),\ \widetilde{B}=\left(\begin{array}{cc} 0 & 2\\
-2& 0\\
-1& 0\\
0 & -1\end{array}\right).$$

An easy calculation shows that the following antisymmetric $4\times
4$ integer
matrix $$\Lambda=\left(\begin{array}{cccc} 0 & 0& 1& 0\\
0 & 0& 0& 1\\
-1 & 0& 0& -2\\
0 & -1& 2& 0\end{array}\right)$$ satisfying
\begin{align}\label{eq:simply_laced_compatible}
\Lambda(-\widetilde{B})=\widetilde{I}:=\begin{bmatrix}I_2\\0
\end{bmatrix},
\end{align}
where $I_2$ is the identity matrix of size $2\times 2$. Then we have
the following result.

\begin{lemma}\label{kro}
Let $R_{p}(1)$ be the indecomposable regular module of degree $1$ as
above. Then
$$X_{R_p(1)}=X_{S_1}X_{S_2}-q^{-\frac{3}{2}}X_1X_2X_3.$$
\end{lemma}
\begin{proof}
By definition, we have
$$X_{S_1}=X^{(-1,2,1,0)}+X^{(-1,0,0,0)};$$
$$X_{S_2}=X^{(0,-1,0,1)}+X^{(2,-1,0,0)};$$
$$X_{R_p(1)}=X^{(-1,1,1,1)}+X^{(1,-1,0,0)}+X^{(-1,-1,0,1)}.$$
Hence the lemma follows from a direct calculation.
\end{proof}
By Lemma \ref{1}, the expression of $X_{R_p(1)}$ is independent of
the choice of  $p\in \mathbb{P}^1_k$ of degree 1. Hence, we set
$$
X_{\delta}:= X_{R_p(1)}.
$$

\begin{remark}\label{kro2}

\begin{enumerate}
  \item By Lemma \ref{kro}, we know that $X_{\delta}$ belongs to $\Acal_{|k|}(\widetilde{Q}).$
 \item By the following Theorem \ref{affine}, $\mathcal{B}(Q)$ is a
$\mathbb{ZP}-$basis in the quantum cluster algebra
$\Acal_{|k|}(\widetilde{Q}).$ Moreover, if specializing $q$ and
coefficients to $1$, $\mathcal{B}(Q)$ is exactly the generic basis
in the sense of \cite{Dup}.
\end{enumerate}
\end{remark}
Note that there is an alternative choice of $(\Lambda,
\widetilde{B})$, i.e.,
$\widetilde{Q}=Q$ and set $\Lambda=\left(\begin{array}{cc} 0 & 1\\
-1 &
0\end{array}\right)$ and $\widetilde{B}=\left(\begin{array}{cc} 0 & 2\\
-2 & 0\end{array}\right)$. Then we have $
\Lambda(-\widetilde{B})=\left(
                          \begin{array}{cc}
                            2 & 0 \\
                            0 & 2 \\
                          \end{array}
                        \right).
$ Hence, one should consider the category of $KQ$-representations
for a field K with $|K|=q^2.$ In this way, we obtain a quantum
cluster algebra of Kronecker type without coefficients. The
multiplication and the bar-invariant bases in this algebra have been
thoroughly studied in \cite{DX}. Moreover, under the specialization
$q=1$, the bases in \cite{DX} induce the canonical basis,
semicanonical basis and generic basis of the cluster algebra of the
Kronecker quiver in the sense of \cite{sherzel},\cite{calzel} and
\cite{Dup}, respectively.

\subsection{The case in  affine types}
An \emph{affine quiver}\index{affine quiver} is an acyclic quiver
whose underlying diagram in an extended Dynkin diagram. One can
refer to \cite{dlab}\cite{CB:lectures}\cite{ringel:1099} for the
theory of representations of affine quivers.   We recall some useful
background concerning representation theory of affine quivers. In
this section we always assume that $Q$ is an affine quiver with
trivial valuation. The category $rep(kQ)$ of finite-dimensional
representations can be identified with the category of mod-$kQ$ of
finite-dimensional modules over the path algebra $kQ.$ It is
well-known that indecomposable $kQ$-module contains (up to
isomorphism) three families: the component of indecomposable regular
modules $\mathcal R(Q)$, the component of the preprojective modules
$\mathcal P(Q)$ and the component of the preinjective modules
$\mathcal I(Q)$. If $P \in \mathcal P(Q)$, $I \in \mathcal I(Q)$ and
$R \in \mathcal R(Q)$, then
    $$\Hom_{kQ}(R,P) \simeq \Hom_{kQ}(I,R) \simeq \Hom_{kQ}(I,P)=0,$$
    and
    $$\Ext^1_{kQ}(P,R)\simeq \Ext^1_{kQ}(R,I)\simeq \Ext^1_{kQ}(P,I)=0.$$
    If $M$ and $N$ are two regular indecomposable modules in different tubes, then
    $$\Hom_{kQ}(M,N)=0 \textrm{ and } \Ext^1_{kQ}(M,N)=0.$$

 There are at most
$3$ non-homogeneous tubes for $Q.$ We denote these tubes by
$\mathcal{T}_1, \cdots, \mathcal{T}_t$. Let $r_i$ be the rank of
$\mathcal{T}_i$ and the regular simple modules in $\mathcal{T}_i$ be
$E^{(i)}_{1}, \cdots E^{(i)}_{r_i}$ such that $\tau
E^{(i)}_2=E^{(i)}_1, \cdots, \tau E^{(i)}_1=E^{(i)}_{r_i}$ for $i=1,
\cdots, t$. If we restrict the discussion to one tube, we will omit
the index $i$ for convenience. Given a regular simple module $E$ in
a  tube, $E[i]$ is the indecomposable regular module with
quasi-socle $E$ and quasi-length $i$ for any $i\in \mathbb{N}$.

Define the set
$$
\textbf{D}(Q)=\{\underline{d}\in \mathbb{N}^ {Q_0}\mid  \exists
\mbox{ a regular rigid module}\ R\   \mbox{and regular simple
modules}\ E_1,\cdots,E_r$$$$ \mbox{in a non-homogenerous tube with
rank}\ r \mbox{ such that }
\mathrm{\underline{dim}}((E_1\oplus\cdots\oplus E_r)^{n}\oplus
R)=\underline{d}.
$$
 Set $\textbf{E}(Q)=\{\underline{d}\in
\mathbb{Z}^{Q_0}\mid \exists M=M_0 \oplus P_M[1]\ \text{with}\ M_0\
\text{is}\  kQ\text{-module},\  P_M\ \text{projective}\
k\widetilde{Q}\text{-module},$
 $M\ \text{rigid object in}\ \mathcal C_{\widetilde{Q}}\  \text{with}\ \mathrm{\underline{dim}} M=\underline{d}\}$. By the main theorem in
\cite{DXX}, we have that $\mathbb{Z}^{Q_0}$ is the disjoint union of
$\textbf{D}(Q)$ and $\textbf{E}(Q)$. We make an assignment, i.e., a
map $$\phi: \mathbb{Z}^{Q_0}\rightarrow
\mathrm{obj}(\mathcal{C}_{\widetilde{Q}})$$ and set
$$X_{\phi(\underline{d})}:=(X_{E_1}\cdots X_{E_r})^{n}X_{R}$$ if $\underline{d}\in\textbf{D}(Q)$ and $|Q_0|>2;$

$$
X_{\phi(\underline{d})}:=X^{n}_{\delta}
$$ for some $\delta$ in a homogeneous tube of degree $1$ if $\underline{d}\in\textbf{D}(Q)$ and $Q$ is the Kronecker
quiver;
$$
X_{\phi(\underline{d})}:=X_{M}
$$ if $\underline{d}\in\textbf{E}(Q)$. It is clear that the above assignment is not
unique. For simplicity and without confusion, we omit $\phi$ in the
notation $X_{\phi(\underline{d})}$.

\begin{theorem}\label{affine}
Let $Q$ be an affine quiver with $Q_0=\{1,2,\cdots, n\}$ and fix an
assignment as above. Then the set
$$\mathcal{B}(Q):=\{X_{\underline{d}}|\underline{d}\in \mathbb{Z}^{Q_0}\}$$
is a $\mathbb{ZP}$-basis of $\Acal_{|k|}(\widetilde{Q})$.
\end{theorem}
\begin{proof}
By \cite{Dup}, there exists an orientation such that $Q'$ is  a
graded quiver where $Q'$ is reflection-equivalent to $Q$.

When $Q'$ is a Kronecker quiver, by Remark \ref{kro2}, we know that
 $X^{n}_{\delta}(n\in\mathbb{N})$  is in $\Acal_{|k|}(\widetilde{Q}')$. If $Q'$ is not
 a Kronecker quiver, we consider the
non-homogeneous tubes. By Theorem \ref{hall multi}, $X_{R}$ is in
$\Acal_{|k|}(\widetilde{Q}')$. Thus $(X_{E_1}\cdots
X_{E_r})^{n}X_{R}$ is in $\Acal_{|k|}(\widetilde{Q}')$.
 Note that for any $\underline{m}=(m_1,\cdots,m_n)\in \mathbb{Z}^n$, $X_{\underline{m}}\in\mathcal{B}(Q')$.  Then by  Proposition \ref{10} we have
$$X^{\widetilde{Q}'}_{\underline{m}}=b_{\underline{m}}\prod_{i=1}^{n}(X_{S_i}^{\widetilde{Q}'})^{m^+_i}(X^{\widetilde{Q}'}_{P_i[1]})^{m^-_i}
+\sum_{\epsilon(\underline{l})<
\epsilon(\underline{m})}b_{\underline{l}}\prod_{i=1}^{n}(X_{S_i}^{\widetilde{Q}'})^{l^+_i}(X^{\widetilde{Q}'}_{P_i[1]})^{l^-_i}$$
where  $b_{\underline{m}},b_{\underline{l}}\in \mathbb{ZP}$. As $Q'$
is a graded quiver, then by Proposition \ref{prop:supportconeCK1},
Proposition \ref{prop:graduationCK1}, it follows that
$b_{\underline{d}}$ must be some nonzero monomial in
$\{q^{\pm\frac{1}{2}},X^{\pm 1
            }_{n+1},\cdots,X^{\pm 1
            }_{m}\}$. Therefore,  $\mathcal{B}(Q')$ is a $\mathbb{ZP}$-basis of
$\Acal_{|k|}(\widetilde{Q}')$. By Theorem \ref{ref}, we obtain that
$\mathcal{B}(Q)$ is a $\mathbb{ZP}$-basis of
$\Acal_{|k|}(\widetilde{Q})$.
\end{proof}

By \cite[Proposition 5]{CR}, the quiver Grassmannians
$\mathrm{Gr}_{\underline{e}}(M)$ of a $kQ$-module $M$ are some
polynomials in $\mathbb{Z}[q].$ Then by specializing $q$ and
coefficients to $1$, the bases in Theorem \ref{affine} induces the
integral bases in affine cluster algebras (\cite{DXX}).
\begin{remark}
Theorem \ref{affine} does not provide the quantum version for
generic bases of affine type in \cite{Dup}. In order to achieve it,
one need to prove a quantum analogue of the difference property
\cite[Definition 3.24]{Dup}.
\end{remark}
\section*{Acknowledgements}
 The authors would like
 to  thank Professor Jie Xiao and Dr. Fan Qin for many helpful discussions and suggestions.

%\section*{Acknowledgments}

\end{document}